\documentclass[12pt]{amsart}
\usepackage{amssymb}
\usepackage{tikz}
\usepackage{cite}
\usepackage{hyperref}
\usepackage{enumerate}
\newcommand{\BG}[1]{\mathbb{B}_{#1}}
\newcommand{\bq}{\boldsymbol{q}}
\newcommand{\bX}{\mathbb{X}}
\newcommand{\cO}{\mathcal{O}}
\newcommand{\End}{\mathrm{End}}
\newcommand{\dlm}{\underline{m}}
\newcommand{\dll}{\underline{l}}
\newcommand{\dle}{\underline{e}}
\newcommand{\fie }{\Bbbk}

\newcommand{\id }{\mathrm{id}}
\newcommand{\lelex }{\le _{\mathrm{lex}}}
\newcommand{\NA }{\mathcal{B}}
\newcommand{\ndN }{\mathbb{N}}
\newcommand{\ndZ }{\mathbb{Z}}
\newcommand{\olV}{\overline{V}}
\newcommand{\ord}{\mathrm{ord}}
\newcommand{\ot }{\otimes}

\newtheorem{theo}{Theorem}[section]
\newtheorem{prop}[theo]{Proposition}
\newtheorem{lemm}[theo]{Lemma}
\newtheorem{coro}[theo]{Corollary}
\theoremstyle{definition}
\newtheorem{defi}[theo]{Definition}
\newtheorem{exam}[theo]{Example}
\newtheorem{rema}[theo]{Remark}

\begin{document}
\title[Nichols algebras which are free algebras]
 {A characterization of Nichols algebras of diagonal type which are free algebras}
\author{I. Heckenberger}
\address{Philipps-Universit\"at Marburg,
FB Mathematik und Informatik,
Hans-Meerwein-Stra\ss e,
35032 Marburg, Germany.}
\email{heckenberger@mathematik.uni-marburg.de}

\author{Y. Zheng}
\address{Department of mathematics East China Normal University, Shanghai 200241, China.}
\email{52150601007@ecnu.cn}
\thanks{The second named author was supported by China Scholarship Council}

\date{}

\begin{abstract}
This paper is devoted to explore the freeness of Nichols algebras of diagonal type and to determine
the dimension of the kernel of the shuffle map considered as an operator acting on the free algebra.
Our proof is based on an inequality for the number of Lyndon words and on an identity for the shuffle map.
For a particular family of examples, the freeness of the Nichols algebra is characterized in terms of
solutions of a quadratic diophantine equation.

\end{abstract}
\maketitle

\section{Introduction}

Since their introduction in the late 70ies by W.~Nichols \cite{MR0506406},
the theory of Nichols algebras enjoyed increasing interest because of its
deep interrelation to different research areas. For an overview we refer to \cite{MR3728608}.
The strongest results have been
obtained for finite-dimensional Nichols algebras of diagonal type, mainly due to
the existence of the root system which was introduced in \cite{H2006}, based on
deep results of V.~Kharchenko \cite{MR1763385} on the structure of certain Hopf algebras
generated by group-like and skew-primitive elements.

A general, very difficult question is, what are the roots and their multiplicities
of a given Nichols algebra of diagonal type. In the case of finite-dimensional
Nichols algebras the answer is known: The roots are the real roots with respect to
the action of the Weyl groupoid, and their multiplicity is one.
The other extreme case is the one of the free algebra, where the root vectors are
parametrized by Lyndon words and appropriate powers of them.
Roots of the form $m\alpha_1+\alpha_2$ with $m\ge 0$ are determined using Rosso's
lemma \cite{MR1632802}.
Roots of the form $m\alpha_1+2\alpha_2$ and their multiplicities have been
determined by the authors in \cite{H2018root}. In this paper we address the question
when the multiplicity of a root is smaller than in the tensor algebra.
In particular, we provide a criterion to decide whether
a given Nichols algebra of diagonal type is a free algebra in terms of
polynomial equations for the entries of the braiding matrix.

The defining ideal of a Nichols algebra is spanned by the kernels of
the braided symmetrizer \cite{schauenburg1996}, which decomposes into
a product of shuffle maps. In \cite{duchamp1997},
the authors study identities involving shuffle maps. We use these identities
to study the freeness of Nichols algebras of diagonal type and to determine
the dimension of the kernel of the shuffle map.
With our results we relate the freeness of Nichols algebras of diagonal
type with braiding matrix $(q^{m_{ij}})_{1\le i,j\le n}$, $m_{ij}\in \ndZ$
for all $i,j$, to solutions of a diophantine equation.

In Section~\ref{basic}
we define a family $(P_{\dlm})_{\dlm\in \ndN_0^n,|\dlm|\ge 2}$ of elements
in the polynomial ring $\ndZ[p_{ij}\mid 1\le i,j\le n]$, where
$|(m_1,\dots,m_n)|=\sum_{i=1}^nm_i$.
Let now $\NA(V)$ be a Nichols algebra of diagonal type of rank $n$ with braiding matrix
$\bq =(q_{ij})_{1\le i,j\le n}\in (\fie^\times)^{n\times n}$, where $\fie $ is a field.

\begin{theo} (see Theorem \ref{theo})
We have $\NA(V)=T(V)$ if and only if $P_{\dlm}(\bq )\ne 0$ for all $\dlm\in \ndN_0^n$
with $|\dlm|\ge 2$.
\end{theo}

Assume that $\fie $ has characteristic $0$. If $P_{\dlm}(\bq)=0$ for some
$\dlm \in \ndN_0^n$, then in Section~\ref{upperbound}
two numbers $n_1(\bq),n_2(\bq)\in \ndN_0$ are defined.

\begin{theo} (see Theorem \ref{theo dim})
Assume that $\fie $ has characteristic $0$.
Let $\dlm =(m_1,\dots,m_n)\in \ndN_0^n$ and $m=\sum_{i=1}^nm_i$ such that
$m\ge 2$, $P_{\dlm}(\bq)=0$, and $P_{\dll }(\bq)\ne0$ for all $\dll<\dlm$.
Then
\begin{align*}
\dim(\ker(\rho_m(S_{1,m-1})|V_{\dlm}))=n_1(\bq)-n_2(\bq).
\end{align*}
\end{theo}

The organization of the paper is as follows.
In Section \ref{basic}, we recall some basic notions about Lyndon words and introduce notations.
We also recall the inequalities on the number of Lyndon words, which the paper is based on.
In Section \ref{shuffles}, we discuss the notion of a free prebraided module over a commutative ring and
compute the determinant of the shuffle map.
In Section \ref{freeness}, we formulate and prove our first main theorem.
In Section \ref{upperbound}, we determine an upper bound for the dimension of the kernel of shuffle map.
In Section \ref{dim}, we prove that this upper bound is a lower bound.

The paper was written during the visit of the second named author to Marburg University
supported by China Scholarship Council. The second named author thanks the department of
FB Mathematik and Informatik of Marburg University for hospitality.

\section{Basic Definitions and properties}
\label{basic}

Throughout this paper we write $\ndN$ and $\ndZ$ for the set of positive integers
and the set of integers, respectively. Let $\ndN_0=\ndN\cup\{0\}$.

We start with recalling necklaces and Lyndon words, and collect some notations.

Let $n\in \ndN$. For any $\dlm=(m_1,\dots,m_n)\in \ndN_0^n$ we write
$|\dlm|=\sum_{i=1}^nm_i$. If additionally $\dlm\ne 0$, then let
$\gcd(\dlm)$ be the greatest common divisor of $m_1,\dots,m_n$,
and if $|\dlm|\ge 2$, then let
$$N(\dlm)=\gcd\{m_i(m_i-1),m_jm_k \mid 1\le i,j,k\le n,j<k\}. $$
For any $1\le i\le n$ let
$\underline{e}_i=(\delta_{ij})_{1\le j\le n}\in \ndN_0^n$,
and for any $k\in \ndN_0$ and any
$\dlm =(m_1,\dots,m_n)\in \ndN_0^n\setminus \{0\}$
let $\dlm/k=(m_1/k,m_2/k,\ldots,m_n/k)$.

There is a partial ordering on $\ndN_0^n$
denoted by $\le$: $\dlm\le\dll$ if and only if $m_i\le l_i$ for all $1\le i\le n$.

Let $B$ be a set (called the alphabet) of $n$ elements denoted by
$b_1,b_2,\ldots,b_n$,
and let $\mathbb{B}$ and $\mathbb{B}^{\times}$ be the set of words and
non-empty words, respectively, with letters in $B$.
For $w=b_{i_1}b_{i_2}\cdots b_{i_s} \in \mathbb{B}$,
in which $b_j$ occurs  $m_j$ times, $1\le j\le n$, we write
 $\deg w=(m_1,m_2,\ldots,m_n)$ and call $\deg w$ the \textbf{degree} of $w$.

We fix a total order $\le $ on $B$. There is a total order $\lelex $
on $\mathbb{B}$ induced by $\le $, called the lexicographic order: For $u,v\in \mathbb{B}$,
one lets $u\lelex v$ if and only if either $v=uw$ for some $w\in \mathbb{B}$,
or there exist $w,u',v'\in \mathbb{B}$ and $x,y\in B$ such that $u=wxu'$, $v=wyv'$,
$x\le y$, and $x\ne y$.

A word $w\in\mathbb{B}^{\times}$ is called a \textbf{necklace} if for any decomposition
$w=uv$ with $u,v\in\mathbb{B}^{\times}$, $w \lelex vu$.
A word $w\in\mathbb{B}^{\times}$ is \textbf{Lyndon} if for any decomposition
$w=uv$, $u,v\in \mathbb{B}^{\times}$, $w \lelex v$.
For any $\dlm \in \ndN_0^n$ let $N_{\dlm }$ and $\ell_{\dlm}$ denote the number
of necklaces and Lyndon words, respectively, of degree $\dlm $.

\begin{rema}\label{reLy}
Any Lyndon word is a necklace, and for any necklace $w$ there is a unique pair
$(v,k)\in \mathbb{B}\times \ndN $ such that $v$ is Lyndon and $w=v^k$. Thus,
for any $\dlm \in \ndN_0^n\setminus\{0\}$,
\begin{align}\label{NL}
N_{\dlm}=\sum_{d\mid\gcd(\dlm)}\ell_{\dlm/d}.
\end{align}
\end{rema}

\begin{rema} \label{re:smallell}
  In\cite{HJ2018} and \cite{Joe1999} one can find explicit formulas
  for $N_{\dlm}$ and $\ell_{\dlm}$ for any $\dlm \in \ndN_0^n$.
  In particular,
  \begin{align}
  \label{eq:numbL1}
  &\ell_{\dle_i+k\dle_j}=1, ~\text{for all}~ k\in \ndN_0;\\
  \label{eq:numbL2}
  &\ell_{k\dle_j}=\delta_{k,1},~\text{for all}~ k\in \ndN_0.
  \end{align}
\end{rema}

In the remaining part of this section we will introduce and study
some polynomials, which are crucial for the paper.
For any ring $R$ and any $q\in R$ let $(0)_q=0$ and
$(m)_q=1+q+\cdots+q^{m-1}$ for any $m\in \ndN$.

\begin{defi}\label{def:Pm}
 For any $\dlm=(m_1,m_2,\ldots,m_n) \in \ndN_0^n$ with $|\dlm|\ge 2$
 let $P_{\dlm}\in \ndZ[p_{ij}\mid 1\le i,j\le n]$ be as follows:
 \begin{enumerate}
  \item If $\dlm=m_i\dle_i$, where $1\le i\le n$, and $m_i\in \ndN$, let
    $$P_{\dlm}=(m_i)_{p_{ii}};$$
  \item If $\dlm=\dle_i+m_j\dle_j$, where $1\le i,j\le n$, $i\ne j$,
  $m_j\in \ndN$, let
    $$P_{\dlm}=1-p_{jj}^{m_j-1}p_{ij}p_{ji};$$
  \item If $\dlm=2\dle_i+m_j\dle_j$, where $1\le i,j\le n$, $i\ne j$,
    $m_j\in \ndN$, let
    $$P_{\dlm}=1+p_{jj}^{m_j(m_j-1)/2}(-p_{ij}p_{ji})^{m_j}p_{ii};$$
  \item If $\dlm=3\dle_i+3\dle_j$, where $1\le i,j\le n$, $i\ne j$, let
       $$P_{\dlm}=(3)_{p_{ii}^2(p_{ij}p_{ji})^3p_{jj}^2};$$
  \item If $\dlm=3\dle_i+4\dle_j$, where $1\le i,j\le n$, $i\ne j$, let
  $$P_{\dlm}=(1-p_{ii}^2(p_{ij}p_{ij})^4p_{jj}^4)(3)_{p_{ii}(p_{ij}p_{ji})^2p_{jj}^2};$$
  \item If $\dlm=3\dle_i+6\dle_j$, where $1\le i,j\le n$, $i\ne j$, let
  $$P_{\dlm}=(1-p_{ii}(p_{ij}p_{ji})^3p_{jj}^{5})(3)_{p_{ii}^{2}(p_{ij}p_{ji})^6p_{jj}^{10}};$$
  \item If $\dlm=4\dle_i+4\dle_j$, where $1\le i,j\le n$, $i\ne j$, let
  $$P_{\dlm}=(1+p_{ii}^{3}(p_{ij}p_{ji})^4p_{jj}^{3})(1+p_{ii}^{6}(p_{ij}p_{ji})^8p_{jj}^{6});$$
  \item Otherwise, let
   $$P_{\dlm}=1-\prod_{1\le i\le n}p_{ii}^{m_i(m_i-1)}
   \prod_{1\le i<j\le n}(p_{ij}p_{ji})^{m_im_j}.$$
 \end{enumerate}
  Moreover, let
  $$ Q_{\dlm}=\prod_{1\le i\le n}p_{ii}^{m_i(m_i-1)/N(\dlm)}
  \prod_{1\le i<j\le n}(p_{ij}p_{ji})^{m_im_j/N(\dlm)}.$$
\end{defi}

\begin{rema}\label{Qfactor}
Let $\dlm \in \ndN_0^n$ with $|\dlm|\ge 2$.
By definition of $N(\dlm)$,
$Q_{\dlm}$ is a well-defined non-constant
monomial in $\ndZ[p_{ij}\mid 1\le i,j\le n]$,
and $Q_{\dlm}$ is not a non-trivial power of any other monomial.
Moreover,
$P_{\dlm}$ divides $1-Q_{\dlm}^{N(\dlm)}$. In particular,
$P_{\dlm}=1-Q_{\dlm}^{N(\dlm)}$ in the last case of Definition~\ref{def:Pm}.
\end{rema}

For any $i\in \ndN$ let $\Phi_i\in \ndZ[x]$ denote the $i$-th cyclotomic
polynomial, that is, the minimal polynomial of any primitive $i$-th root of $1$
in the complex numbers. Clearly, $x^k-1=\prod_{i\mid k}\Phi_i$ for any $k\in
\ndN$.

Next we describe the irreducible factors of the polynomials
$P_{\dlm}$.

\begin{lemm} \label{le:mmatrix}
	Let $D$ be a Euclidean domain, let
	$\dlm =(m_1,\dots,m_n)$ be a non-zero vector in $D^n$, and let
	$d=\gcd(m_1,\dots,m_n)$. Then there is a matrix $M\in D^{n\times n}$
	with $\dlm$ as its first row and determinant $d$.
\end{lemm}

\begin{proof}
	View $\dlm$ as a $1\times n$-matrix. Choose a composition $f$
	of elementary column transformations which maps $\dlm$
	to the vector $(d,0,\dots,0)$. Let $M'$ be the diagonal matrix with
	diagonal entries $(d,1,\dots,1)$. Then $M=f^{-1}(M')$ satisfies the desired
	properties.
\end{proof}

\begin{rema}\footnote{Thanks to Ben Anthes for this remark.}
  Lemma\ref{le:mmatrix} also holds for principal ideal domains $D$.
  On the other hand, let $D=\fie [x,y]$ for some field $\fie $
  and let $\dlm=(x,y)$. Then $\gcd(\dlm)=1$, but there are no $a,b\in D$
  with $xb-ya=1$. Hence Lemma~\ref{le:mmatrix} does not hold for this $D$.
\end{rema}

\begin{lemm} \label{le:Laurentauto}
	Let $\dlm=(m_1,\dots,m_n)$ be a non-zero vector in $\ndZ^n$ with
	$\gcd(m_1,\dots,m_n)=1$.
	Then there is a ring automorphism $\varphi $ of the Laurent polynomial ring
	$\ndZ[x_1^{\pm 1},\dots,x_n^{\pm 1}]$
	with $\varphi (x_1)=x_1^{m_1}\cdots x_n^{m_n}$.
\end{lemm}

\begin{proof}
	By Lemma~\ref{le:mmatrix} there is a matrix $M\in \ndZ^{n\times n}$ with
	$\dlm$ as its first row and with determinant $1$. Then
	$\varphi (x_i)=x_1^{m_{i1}}x_2^{m_{i2}}\cdots x_n^{m_{in}}$ for $1\le i\le n$
	defines a ring automorphism of
	$\ndZ[x_1^{\pm 1},\dots,x_n^{\pm 1}]$ as desired.
\end{proof}

\begin{lemm} \label{le:PhikQirred}
	For any $k\in \ndN$ and any $\dlm \in \ndN_0^n$ with $|\dlm|\ge 2$,
	the polynomial $\Phi_k(Q_{\dlm})\in \ndZ[p_{ij}\mid 1\le i,j\le n]$
	is irreducible. In particular, $Q_{\dlm}^k-1=\prod_{i\mid k}\Phi_i(Q_{\dlm})$
	is the unique factorization of $Q_{\dlm}^k-1$ into irreducibles, and each
	irreducible factor of $P_{\dlm}$ is of the form $\Phi_l(Q_{\dlm})$ for some
	$l\mid N(\dlm)$.
\end{lemm}

\begin{proof}
	Let $k\in \ndN$ and $\dlm \in \ndN_0^n$ with $|\dlm|\ge 2$.
	For any $1\le i,j,l\le n$ let
	$$m_{ii}=m_i(m_i-1)/N(\dlm),\quad m_{jl}=m_jm_l/N(\dlm).$$
	Then $\gcd(m_{ij}\mid 1\le i,j\le n)=1$ and $Q_{\dlm}=\prod_{1\le i,j\le
	n}p_{ij}^{m_{ij}}$
	by construction. By Lemma~\ref{le:Laurentauto}
	there is a ring automorphism $\varphi $ of the Laurent polynomial ring
	$\ndZ[p_{ij}^{\pm 1}\mid 1\le i,j\le n]$ with $\varphi (p_{11})=
	Q_{\dlm}$.
	Thus $\Phi_k(Q_{\dlm})=\varphi (\Phi_k(p_{11}))$ is irreducible in
	$\ndZ[p_{ij}^{\pm 1}\mid 1\le i,j\le n]$.
	Since $\Phi_k(Q_{\dlm})$ is not divisible in
	$\ndZ[p_{ij}\mid 1\le i,j\le n]$ by any $p_{ij}$ with $1\le i,j\le n$,
	the polynomial $\Phi_k(Q_{\dlm})$ is irreducible.
\end{proof}

\begin{lemm}\label{le:prime}
  Let $\dlm,\dll\in \ndN_0^n$ with $|\dll|\ge 2$.
Suppose that there exist $1\le i<j\le n$ such that $m_i,m_j\ne 0$.
Then $P_{\dlm}$ and $P_{\dll}$
are relatively prime if and only if $\dlm\ne\dll$.
In particular,  $P_{\dlm}$ and $P_{\dll}$
are relatively prime whenever $\dll<\dlm$.
\end{lemm}

\begin{proof}
	Recall that $P_{\dlm}$ is not constant. Thus, if $P_{\dlm}$
  and $P_{\dll}$ are relatively prime, then $\dlm\ne\dll$.

  Conversely, suppose that $P_{\dlm}$ and $P_{\dll}$ are not relatively prime.
  Then, by Remark~\ref{Qfactor}, $Q_{\dlm}^{N(\dlm)}-1$ and
	$Q_{\dll}^{N(\dll)}-1$
	are not relatively prime. Let $f$ be a non-constant
	common factor of $Q_{\dlm}^{N(\dlm)}-1$ and $Q_{\dll}^{N(\dll)}-2$.
  Lemma~\ref{le:PhikQirred} implies that there exist non-constant
	monic polynomials $p_1,p_2\in \ndZ[x]$ with $f=p_1(Q_{\dlm})=p_2(Q_{\dll})$.
	In particular, $Q_{\dlm}=Q_{\dll}$.
  Let $1\le i<j\le n$ with $m_i,m_j\ne 0$. Then $l_i,l_j\ne 0$, and
	the following equations hold:
\begin{align}
\label{Eq1}
&\frac{m_i(m_i-1)}{N(\dlm)}=\frac{l_i(l_i-1)}{N(\dll)},\\
\label{Eq2}
&\frac{m_j(m_j-1)}{N(\dlm)}=\frac{l_j(l_j-1)}{N(\dll)},\\
\label{Eq3}
&\frac{m_im_j}{N(\dlm)}=\frac{l_il_j}{N(\dll)}.
\end{align}
From Equation \eqref{Eq1} and \eqref{Eq3}, one gets
\begin{align*}
l_jm_i-l_im_j=l_j-m_j.
\end{align*}
Similarly, using Equation \eqref{Eq2} and \eqref{Eq3}, one gets
\begin{align}
\label{Eq4}
l_jm_i-l_im_j=m_i-l_i.
\end{align}
Thus $m_i+m_j=l_i+l_j$. Let $t=m_i+m_j=l_i+l_j$. Replacing $m_j$ with $t-m_i$ and
$l_j$ with $t-l_i$ in Equation \eqref{Eq4}, we get
$t(m_i-l_i)=m_i-l_i$, and hence $m_i=l_i$ because of $t>1$. Thus
$m_j=l_j$. It follows that $\dlm=\dll$.
\end{proof}

Now we pass to another family of polynomials, which are the main reason for our
interest in the family $(P_{\dlm})_{\dlm \in \ndN_0^n,|\dlm|\ge 2}$.

\begin{defi} \label{defAm}
For any $\dlm=(m_1,m_2,\dots,m_n)\in \ndN_0^n$, $|\dlm|\ge 2$, let
$$A_{\dlm}=\frac{\prod_{i:m_i>0}\prod_{k\mid \gcd(\dlm-\dle_i)}
(1-Q_{\dlm}^{N(\dlm)/k})^{\ell_{(\dlm-\dle_i)/k}}}
{\prod_{k\mid \gcd(\dlm)}(1-Q_{\dlm}^{N(\dlm)/k})^{\ell_{\dlm/k}}}.
$$
\end{defi}

Note that any $k$ with $k\mid \gcd(\dlm-\dle_i)$ divides $m_i-1$ and any $m_j$,
$j\ne i$, and hence $k\mid N(\dlm)$. Similarly, $k\mid \gcd(\dlm)$ implies that
$k\mid N(\dlm)$. Therefore the numerator and the denominator of $A_{\dlm}$ are
polynomials. If $\dlm=m_i\dle_i$ for some $1\le i\le n$ and $m_i\ge 2$, then
$Q_{\dlm}=p_{ii}$ and
\begin{align} \label{eq:Amiei}
	A_{m_i\dle _i}=\frac{1-Q_{\dlm}^{N(\dlm)/(m_i-1)}}{1-Q_{\dlm}^{N(\dlm)/m_i}}
	=\frac{1-Q_{\dlm}^{m_i}}{1-Q_{\dlm}^{m_i-1}}=\frac{(m_i)_{p_{ii}}}{(m_i-1)_{p_{ii}}}.
\end{align}
In order to show that every other $A_{\dlm}$ is a polynomial, we use
some results in \cite{HJ2018} about the number of Lyndon words.
Using Equation~\eqref{NL}, these can be restated as follows.

\begin{theo}(\cite[Lemma~4.1,~4.2,~Theorem~1.2]{HJ2018})\label{basic theo}
Let $\dlm\in \ndN_0^n$. Assume that $m_s,m_t\ne0$ for some $1\le s<t\le n$. Then,
\begin{enumerate}
\item
\begin{align*}
  \sum_{k\mid\gcd(\dlm)}\ell_{\dlm/k}\le\sum_{1\le i\le n,m_i>0}\ell_{\dlm-\dle_i}.
\end{align*}
\item
\begin{align}
&\sum_{k\mid \gcd(\dlm)}\ell_{\dlm/k}=\sum_{1\le i\le
n,m_i>0}\ell_{\dlm-\dle_i}
\end{align}
if and only if $\dlm$ is one of the cases $(2)$,$(3)$,$(4)$,$(5)$,$(6)$,$(7)$ in
Definition~\ref{def:Pm} different from $\dlm=\dle_s+\dle_t$.
\end{enumerate}
\end{theo}

\begin{lemm}\label{le:Am}
 Let $\dlm=(m_1,m_2,\dots,m_n)\in \ndN_0^n$.
 If there exist $i,j\in\{1,2,\ldots,n\}$, $i\ne j$, such that $m_i\ne0$, $m_j\ne0$, then
  $A_{\dlm}$ is a polynomial in $\ndZ[p_{ij},1\le i,j\le n]$.
\end{lemm}

\begin{proof}
	The numerator of $A_{\dlm}$ is a multiple of
	$\prod_{i:m_i>0}(1-Q_{\dlm}^{N(\dlm)})^{\ell_{\dlm-\dle_i}}$
	and the denumerator of $A_{\dlm}$ is a divisor of
	$\prod_{k\mid \gcd(\dlm)}(1-Q_{\dlm}^{N(\dlm)})^{\ell_{\dlm/k}}$.
	Thus the claim follows from Theorem~\ref{basic theo}(1).
\end{proof}

\begin{prop} \label{pr:Am=0}
Let $\dlm=(m_1,m_2,\ldots,m_n) \in \ndN_0^n$ with $m_s,m_t>0$ for some $1\le
s<t\le n$. Then $P_{\dlm}$ is the product of the irreducible factors of
$A_{\dlm}$.
\end{prop}

\begin{proof}
	We follow Definition~\ref{def:Pm} case by case to compare $A_{\dlm}$ and
	$P_{\dlm}$. Then the claim follows directly from Lemma~\ref{le:PhikQirred}.

	(1) If $\dlm=m_i\dle_i$ for some $1\le i\le n$, $m_i\ge 2$, then the
	assumptions of the lemma are not fulfilled.

	(2)
	Assume that $\dlm=\dle_i+m_j\dle_j$, $1\le i,j\le n$, $i\ne j$, $m_j>0$.
	Then $N(\dlm)=m_j$, $Q_{\dlm}=p_{ij}p_{ji}p_{jj}^{m_j-1}$, and
\begin{align*}
	A_{\dlm}&=\frac{(1-Q_{\dlm}^{m_j})^{\ell_{\dle_i+(m_j-1)\dle_j}}
	\prod_{k\mid m_j}(1-Q_{\dlm }^{m_j/k})^{\ell_{m_j\dle_j/k}}}
	{(1-Q_{\dlm }^{m_j})^{\ell_{\dlm}}}\\
	&=\frac{(1-Q_{\dlm}^{m_j})(1-Q_{\dlm})}{1-Q_{\dlm}^{m_j}}
	=1-Q_{\dlm}=1-p_{ij}p_{ji}p_{jj}^{m_j-1}
  =P_{\dlm},
\end{align*}
where we used Equations \eqref{eq:numbL1} and \eqref{eq:numbL2}.

(3) Assume that $\dlm=2\dle_i+m_j\dle_j$, $1\le i,j\le n$, $i\ne j$, $m_j>0$.
Then $N(\dlm)=2$ and $Q_{\dlm}=p_{ii}(p_{ij}p_{ji})^{m_j}p_{jj}^{m_j(m_j-1)/2}$.
Moreover,
\begin{align*}
	&A_{\dlm}=\frac{(1-Q_{\dlm}^{N(\dlm)})^{\ell_{\dlm-\dle_i}}
	\prod_{k\mid\gcd(2,m_j-1)}(1-Q_{\dlm}^{N(\dlm)/k})^{\ell_{(\dlm-\dle_j)/k}}}
	{\prod_{k\mid \gcd(2,m_j)}(1-Q_{\dlm}^{N(\dlm)/k})^{\ell_{\dlm/k}}}.
\end{align*}
If $m_j$ is even, then
$\ell_{\dlm/2}+\ell_{\dlm}=\ell_{\dlm-\dle_i}+\ell_{\dlm-\dle_j}$ by
Theorem~\ref{basic theo}(2). Thus
\begin{align*}
A_{\dlm}&=\frac{(1-Q_{\dlm}^2)^{\ell_{\dlm-\dle_i}}
                (1-Q_{\dlm}^2)^{\ell_{\dlm-\dle_j}}}
               {(1-Q_{\dlm}^2)^{\ell_{\dlm}}(1-Q_{\dlm})^{\ell_{\dlm/2}}}
							 =1+Q_{\dlm}=P_{\dlm}
\end{align*}
by Equation \eqref{eq:numbL1}.

If $m_j$ is odd, then
$\ell_{(\dlm-\dle_i)}+\ell_{(\dlm-\dle_j)}=\ell_{\dlm}$ by
Theorem~\ref{basic theo}(2). Therefore
\begin{align*}
A_{\dlm}&=\frac{(1-Q_{\dlm}^2)^{\ell_{\dlm-\dle_i}}
(1-Q_{\dlm}^2)^{\ell_{\dlm-\dle_j}}(1-Q_{\dlm})^{\ell_{(\dlm-\dle_j)/2}}}
               {(1-Q_{\dlm}^2)^{\ell_{\dlm}}}
							 =1-Q_{\dlm}=P_{\dlm}
\end{align*}
by Equation \eqref{eq:numbL1}.

(4)
Assume that $\dlm=3\dle_i+3\dle_j$ with $1\le i<j\le n$. Then $N(\dlm)=3$,
$Q_{\dlm}=p_{ii}^2(p_{ij}p_{ji})^3p_{jj}^2$, and $\ell_{(1,1)}+\ell_{(3,3)}
=\ell_{(2,3)}+\ell_{(3,2)}$ by Theorem~\ref{basic theo}(2). Thus
\begin{align*}
	A_{\dlm}&=\frac{\prod_{k\mid\gcd(2,3)}(1-Q_{\dlm }^{3/k})^{\ell_{(2,3)/k}}
	\prod_{k\mid\gcd(3,2)}(1-Q_{\dlm}^{3/k})^{\ell_{(3,2)/k}}}
	{\prod_{k\mid\gcd(3,3)}(1-Q_{\dlm}^{3/k})^{\ell_{(3,3)/k}}}\\
	&=\frac{(1-Q_{\dlm }^3)^{\ell_{(2,3)}+\ell_{(3,2)}}}
	{(1-Q_{\dlm }^3)^{\ell_{(3,3)}}(1-Q_{\dlm})^{\ell_{(1,1)}}}
	=\frac{(1-Q_{\dlm}^3)^{\ell_{(1,1)}}}{(1-Q_{\dlm})^{\ell_{(1,1)}}}
	=(3)_{Q_{\dlm}}=P_{\dlm}.
\end{align*}

(5)
If $\dlm=3\dle_i+4\dle_j$ with $1\le i,j\le n$, $i\ne j$, then
$N(\dlm)=6$, $Q_{\dlm}=p_{ii}(p_{ij}p_{ji})^2p_{jj}^2$, and
$\ell_{(3,4)}=\ell_{(2,4)}+\ell_{(3,3)}$. Thus
\begin{align*}
	A_{\dlm}&=\frac{\prod_{k\mid\gcd(2,4)}(1-Q_{\dlm}^{6/k})^{\ell_{(2,4)/k}}
	\prod_{k\mid \gcd(3,3)}(1-Q_{\dlm}^{6/k})^{\ell_{(3,3)/k}}}
	{\prod_{k\mid\gcd(3,4)}(1-Q_{\dlm}^{6/k})^{\ell_{(3,4)/k}}}\\
	&=\frac{(1-Q_{\dlm}^6)^{\ell_{(2,4)}}(1-Q_{\dlm}^3)^{\ell_{(1,2)}}
	(1-Q_{\dlm}^6)^{\ell_{(3,3)}}(1-Q_{\dlm}^2)^{\ell_{(1,1)}}}
	{(1-Q_{\dlm}^6)^{\ell_{(3,4)}}}\\
	&=(1-Q_{\dlm}^2)(1-Q_{\dlm}^3)=(1-Q_{\dlm})^2(1+Q_{\dlm})(3)_{Q_{\dlm}}
	=(1-Q_{\dlm})P_{\dlm}.
\end{align*}

(6)
If $\dlm=3\dle_i+6\dle_j$ with $1\le i,j\le n$, $i\ne j$, then
$N(\dlm)=6$, $Q_{\dlm}=p_{ii}(p_{ij}p_{ji})^3p_{jj}^5$, and
$\ell_{(3,6)}+\ell_{(1,2)}=\ell_{(2,6)}+\ell_{(3,5)}$. Thus
\begin{align*}
	A_{\dlm}&=\frac{\prod_{k\mid\gcd(2,6)}(1-Q_{\dlm}^{6/k})^{\ell_{(2,6)/k}}
	\prod_{k\mid \gcd(3,5)}(1-Q_{\dlm}^{6/k})^{\ell_{(3,5)/k}}}
	{\prod_{k\mid\gcd(3,6)}(1-Q_{\dlm}^{6/k})^{\ell_{(3,6)/k}}}\\
	&=\frac{(1-Q_{\dlm}^6)^{\ell_{(2,6)}}(1-Q_{\dlm}^3)^{\ell_{(1,3)}}
	(1-Q_{\dlm}^6)^{\ell_{(3,5)}}}
	{(1-Q_{\dlm}^6)^{\ell_{(3,6)}}(1-Q_{\dlm}^2)^{\ell_{(1,2)}}}
	=\frac{(1-Q_{\dlm}^6)(1-Q_{\dlm}^3)}{1-Q_{\dlm}^2}\\
	&=(1-Q_{\dlm})(3)_{Q_{\dlm}}^2(3)_{-Q_{\dlm}}
	=(3)_{Q_{\dlm}}P_{\dlm}.
\end{align*}

(7)
If $\dlm=4\dle_i+4\dle_j$ with $1\le i,j\le n$, $i\ne j$, then
$N(\dlm)=4$, $Q_{\dlm}=p_{ii}^3(p_{ij}p_{ji})^4p_{jj}^3$, and
$\ell_{(1,1)}+\ell_{(2,2)}+\ell_{(4,4)}=\ell_{(3,4)}+\ell_{(4,3)}$. Thus
\begin{align*}
	A_{\dlm}&=\frac{\prod_{k\mid\gcd(3,4)}(1-Q_{\dlm}^{4/k})^{\ell_{(3,4)/k}}
	\prod_{k\mid \gcd(4,3)}(1-Q_{\dlm}^{4/k})^{\ell_{(4,3)/k}}}
	{\prod_{k\mid\gcd(4,4)}(1-Q_{\dlm}^{4/k})^{\ell_{(4,4)/k}}}\\
	&=\frac{(1-Q_{\dlm}^4)^{\ell_{(3,4)}}(1-Q_{\dlm}^4)^{\ell_{(4,3)}}}
	{(1-Q_{\dlm}^4)^{\ell_{(4,4)}}(1-Q_{\dlm}^2)^{\ell_{(2,2)}}
	(1-Q_{\dlm})^{\ell_{(1,1)}}
  }
	=\frac{(1-Q_{\dlm}^4)^2}{(1-Q_{\dlm}^2)(1-Q_{\dlm})}\\
	&=(1+Q_{\dlm}^2)^2(1+Q_{\dlm})=(1+Q_{\dlm}^2)P_{\dlm}
\end{align*}
since $\ell_{(2,2)}=1$.

(8)
Now we suppose $\dlm$ is not equal to any of the above cases.
By following the proof of Lemma~\ref{le:Am} and using Theorem~\ref{basic theo}
we conclude that $Q_{\dlm}^{N(\dlm)}-1$ divides $A_{\dlm}$. Moreover, every
irreducible factor of $A_{\dlm}$ is a factor of $P_{\dlm}=1-Q_{\dlm }^{N(\dlm)}$.
Since $P_{\dlm}$ is a product of pairwise non-associated irreducible factors,
we conclude that $P_{\dlm}$ contains every irreducible factor of $A_{\dlm}$
precisely once.

Thus the proof is completed.
\end{proof}

\begin{coro}\label{cor:prime}
Let $\dlm,\dll\in \ndN_0^n$ be different from $m\dle_i$ for all $m\ge 0$ and
$1\le i\le n$.
Then $A_{\dlm}$ and $A_{\dll}$
are relatively prime if and only if $\dlm\ne\dll$. In particular,
$A_{\dlm}$ and $A_{\dll}$
are relatively prime whenever $\dll<\dlm$.
\end{coro}

\begin{proof}
The claim follows from Lemma~\ref{le:prime} and Proposition~\ref{pr:Am=0}.
\end{proof}

\section {The shuffle map over commutative rings}
\label{shuffles}

In this section let $R$ be a unital commutative ring. We calculate
the determinant of the shuffle map over $R$.

\begin{defi}
Let $\overline{V}$ be a finitely generated free module over $R$
and let $\bar{c}:\overline{V}\ot _R\overline{V}\rightarrow \overline{V}\ot _R\overline{V}$
be an $R$-module endomorphism of $\overline{V}\ot _R\overline{V}$.
We say that $(\overline{V},\bar{c})$ is a \emph{free prebraided module of diagonal type over} $R$,
if there exist a basis $x_1,x_2,\ldots,x_n$ of $\overline{V}$ and
$(\overline{q}_{ij})_{1\le i,j\le n}\in R^{n\times n}$ with
  $$\bar{c}(x_i\otimes x_j)=\overline{q}_{ij} x_j\otimes x_i \quad \text{for all $i,j$.} $$
\end{defi}

Let $n\in \ndN$ and let $(\overline{V},\overline{c})$ be a free prebraided module
of diagonal type with basis $x_1,\dots,x_n$. Let $I=\{1,2,\ldots,n\}$ and
$(\overline{q}_{ij})_{i,j\in I}\in R^{n\times n}$. Assume that
$$\overline{c}(x_i\otimes x_j)=\overline{q}_{ij} x_j\otimes x_i$$
for all $i,j\in I$. Let $\overline{V}^{\otimes k}$ denote the $k$-fold
tensor product of $\overline{V}$ over $R$ and let
$T(\overline{V})=\bigoplus_{k=0}^{\infty} \overline{V}^{\otimes k}$.
Note that $\overline{V}^{\otimes k}$ is a free module over $R$
for all $k\in \ndN$.

For any $\dlm\in \ndN^n$ let $\mathbb{X}_{\dlm}$ denote the set of words over $I$
of degree $\dlm$.
Let $\alpha_1,\alpha_2,\ldots,\alpha_n$ be the standard basis of $\ndZ^n$. Then $T(\overline{V})$ admits a $\ndZ^n$-grading given by $\deg x_i=\alpha_i$, for all $i\in I$. Thus
for any $i_1i_2\cdots i_l\in \mathbb{X}_{\dlm}$, the degree of
$x_{i_1}x_{i_2}\cdots x_{i_l}$ is $\sum_{j=1}^l \alpha_{i_j}$,
and we write $\deg x$ for the degree
of any homogeneous element $x$ of $T(\overline{V})$.
For any $\dlm \in \ndN_0^n$
let $\olV_{\dlm}$ denote the $\ndZ^n$-homogeneous
component of $T(\overline{V})$ of degree $\dlm$.

For any $k\ge 2$, let $\mathbb{B}_k$ denote the monoid which is generated by generators
$\sigma_1,\sigma_2,\ldots \sigma_{k-1}$ and relations
\begin{enumerate}
  \item $\sigma_i\sigma_j=\sigma_j\sigma_i$ for
  $i,j\in\{1,2,\ldots, k-1\}$ with $|i-j|\ge 2$, and
  \item $\sigma_i\sigma_{i+1}\sigma_i=\sigma_{i+1}\sigma_{i}\sigma_{i+1}$
  for $1\le i\le k-2$.
\end{enumerate}
Let $S_{1,0}=1 \in \mathbb{B}_k$, and for any $1\le m<k$ let
 $$ S_{1,m}=1+\sigma_1+\sigma_2\sigma_1+\cdots
     +\sigma_m\sigma_{m-1}\cdots \sigma_1. $$
 A variant of the following equation appeared already in \cite[Lemma~6.12]{duchamp1997}.
\begin{lemm}\label{basic Lemm}
For any $k>m\ge 1$ the following equation holds in $\BG k$.
\begin{align*}
		(1-\sigma_m\cdots \sigma_2\sigma_1)S_{1,m}
		=S_{1,m-1}(1-\sigma_m\cdots \sigma_2\sigma_1^2).
\end{align*}
\end{lemm}

\begin{proof}
It is easy to check that
$$(\sigma_m\sigma_{m-1}\cdots\sigma_2\sigma_1)\sigma_{i}=
\sigma_{i-1}(\sigma_m\sigma_{m-1}\cdots\sigma_2\sigma_1)
$$
for all $2\le i\le m<k$. Thus
\begin{align*}
&\sigma_m\cdots \sigma_2\sigma_1S_{1,m}\\
&=\sum_{t=1}^m \sigma_m\cdots \sigma_2\sigma_1 \sigma_{t}\sigma_{t-1}\cdots \sigma_2\sigma_1
+\sigma_m\cdots \sigma_2\sigma_1\\
&=\sum_{t=1}^{m-1}\sigma_{t}\sigma_{t-1}\cdots \sigma_2\sigma_1
(\sigma_m\cdots \sigma_2\sigma_1)\sigma_1+\sigma_m\cdots \sigma_2\sigma_1^2+
\sigma_m\cdots \sigma_2\sigma_1\\
&=S_{1,m-1}(\sigma_m\cdots \sigma_2\sigma_1^2)+\sigma_m\cdots \sigma_2\sigma_1.
\end{align*}
Hence,
\begin{align*}
&(1-\sigma_m\cdots \sigma_2\sigma_1)S_{1,m}\\
&=S_{1,m}-S_{1,m-1}(\sigma_m\cdots \sigma_2\sigma_1^2)
-\sigma_m\cdots \sigma_2\sigma_1\\
&=S_{1,m-1}-S_{1,m-1}(\sigma_m\cdots \sigma_2\sigma_1^2)\\
&=S_{1,m-1}(1-\sigma_m\cdots \sigma_2\sigma_1^2).
\end{align*}
This proves the lemma.
\end{proof}

For any $m\ge 2$ let $R\BG m$ denote the monoid ring of $\BG m$ over $R$ and let
$\overline{\rho}_m:R\BG m\rightarrow\End_R(\overline{V}^{\otimes m})$ be the ring
homomorphism such that $\overline{\rho}_m(\sigma_i)$ is the prebraiding
$\overline{c}$ applied to the $i$-th and $i+1$-th tensor factors of
$\overline{V}^{\otimes m}$.

Let $f:\ndZ[p_{ij}\mid i,j\in I]\rightarrow R$ be the ring homomorphism such that
 $$f(p_{ij})=\overline{q}_{ij} \quad \text{for all $i,j\in I$.} $$

\begin{lemm}\label{det1}
For any $\dlm\in \ndN_0^n$ with $\dlm\ne0$ and $m=|\dlm|$ we have
\begin{align*}
\det(\overline{\rho}_m(1-\sigma_{m-1}\cdots\sigma_2\sigma_1)\big|\overline{V}_{\dlm})
  =\prod_{k\mid \gcd(\dlm)}(1-f(Q_{\dlm})^{N(\dlm)/k})^{\ell_{\dlm/k}}.
\end{align*}
Moreover, if $R$ is a field, $d=\mathrm{ord}(f(Q_{\dlm}))$, and $d\mid
N(\dlm)$, then
\begin{align*}
	\dim (\ker
	(\overline{\rho}_m(1-\sigma_{m-1}\cdots\sigma_2\sigma_1)\big|\olV_{\dlm}))
	=\sum_{k\mid \gcd(\dlm),k\mid N(\dlm)/d}\ell _{\dlm/k}.
\end{align*}
\end{lemm}

\begin{proof}
Consider the action of $\ndZ$ on $\mathbb{X}_{\underline{m}}$ given by
   $$1\cdot i_1 \cdots i_m=i_2\cdots i_mi_1.$$
In any $\ndZ$-orbit of $\mathbb{X}_{\underline{m}}$ there is a unique element $v^k$,
where $v$ is a Lyndon word and $k \in \ndN$ with $k|m_i$ for each $1\le i\le n$.

To any $\ndZ$-orbit $\mathcal{O}$ of $\mathbb{X}_{\dlm}$
we attach the submodule $\overline{V}_{\mathcal{O}}$ of $\overline{V}_{\underline{m}}$ generated by
$x_{i_1}x_{i_2}\cdots x_{i_m}$ with $i_1i_2 \cdots i_m \in \mathcal{O}$. Then
\begin{align}\label{decom1}
 \overline{V}_{\underline{m}}=\bigoplus_{\mathcal{O}}\overline{V}_{\mathcal{O}}.
\end{align}
Let $\mathcal{O}$ be a $\ndZ$-orbit and let $k\ge 1$ and $v=i_1i_2\cdots i_l$ be the
Lyndon word such that $v^k\in \mathcal{O}$. Then $l=m/k$. Let
$\widetilde{v}=x_{i_1}x_{i_2}\cdots x_{i_l}\in T(\overline{V})$.
Then
\begin{align} \label{eq:O1basis}
x_{i_t}x_{i_{t+1}}\cdots x_{i_l}\widetilde{v}^{k-1}x_{i_1}\cdots x_{i_{t-1}}, \quad
1\le t\le l,
\end{align}
is a basis of $\overline{V}_{\mathcal{O}}$, where $x_{i_1}\cdots x_{i_0}=1$. Moreover,
for all $1\le t\le l$,
\begin{align*}
&\overline{\rho}_m(1-\sigma_{m-1}\cdots\sigma_2\sigma_1)
(x_{i_t}\cdots x_{i_{l}}\widetilde{v}^{k-1}x_{i_1}\cdots x_{i_{t-1}})\\
&=x_{i_t}\cdots x_{i_l}\widetilde{v}^{k-1}x_{i_1}\cdots x_{i_{t-1}}
    -\lambda_tx_{i_{t+1}}\cdots x_{i_{l}}\widetilde{v}^{k-1}x_{i_1}
   \cdots x_{i_{t}},
\end{align*}
where $\lambda_t=\overline{q}_{i_ti_1}^{k}\overline{q}_{i_ti_2}^{k}\cdots
   \overline{q}_{i_ti_{t-1}}^{k}\overline{q}_{i_ti_{t}}^{k-1}
   \overline{q}_{i_ti_{t+1}}^{k}\cdots \overline{q}_{i_ti_{l}}^{k}$
   and $x_{i_{t+1}}\cdots x_{i_t}=1$.

We obtain that the matrix of $\overline{\rho}_m(1-\sigma_{m-1}\cdots\sigma_2\sigma_1)
\big|\overline{V}_{\mathcal{O}}$ with respect to the basis \eqref{eq:O1basis}
is $A=(a_{st})_{1\le s,t\le l}$, where
$$a_{st}=
\begin{cases}
     1& \text{if $s=t$,}  \\
     -\lambda_t & \text{if $s=t+1$, $1\le t\le l-1$,} \\
     -\lambda_l & \text{if $s=1$, $t=l$,} \\
     0& \text{otherwise.}
\end{cases}
$$
Therefore,
\begin{align*}
&\det(\overline{\rho}_m(1-\sigma_{m-1}\cdots\sigma_2\sigma_1)\big|\overline{V}_{\mathcal{O}})\\
&=1+(-1)^{l+1}(-1)^l\lambda_1\lambda_2\cdots\lambda_l\\
&=1-\prod_{1\le t \le l}\overline{q}_{i_ti_t}^{k-1}\prod_{1\le t<s\le l}(\overline{q}_{i_ti_s}
\overline{q}_{i_si_t})^k\\
&=1-\prod_{1\le t \le n}\overline{q}_{tt}^{m_t(m_t-1)/k}\prod_{1\le t<s\le n}
(\overline{q}_{st}\overline{q}_{ts})^{m_tm_s/k}\\
&=1-f(Q_{\dlm})^{N(\dlm)/k}.
\end{align*}
Hence
\begin{align*}
\det(\overline{\rho}_m(1-\sigma_{m-1}\cdots\sigma_2\sigma_1)\big|\olV_{\dlm})
  =\prod_{k\mid \gcd(\dlm)}(1-f(Q_{\dlm})^{N(\dlm)/k})^{\ell_{\dlm/k}}
\end{align*}
because of the decomposition of $\olV_{\dlm}$ in \eqref{decom1}.

If $R$ is a field, then the matrix $A$ above has corank $0$ or $1$.
Moreover, $A$ has corank $1$ if and only if
$f(Q_{\dlm})^{N(\dlm)/k}=1$, that is, if and only if $d\mid N(\dlm)/k$,
where $d=\mathrm{ord}(f(Q_{\dlm}))$. This implies the last claim.
\end{proof}

\begin{lemm}\label{det2}
For any $\dlm\in \ndN_0^n$ with $m=|\dlm|\ge 2$ we have
\begin{align*}
&\det(\overline{\rho}_m(1-\sigma_{m-1}\cdots\sigma_2\sigma_1^2)\big|\olV_{\dlm})\\
&\quad =\prod_{ i:m_i>0}\prod_{k\mid \gcd(\dlm-\dle_i)}
(1-f(Q_{\dlm})^{N(\dlm)/k})^{\ell_{(\dlm-\dle_i)/k}}.
\end{align*}
Moreover, if $R$ is a field, $d=\mathrm{ord}(f(Q_{\dlm}))$, and $d\mid
N(\dlm)$, then
\begin{align*}
	\dim (\ker
	(\overline{\rho}_m(1-\sigma_{m-1}\cdots\sigma_2\sigma_1^2)\big|\olV_{\dlm}))
	=\sum_{i:m_i>0}\sum_{k\mid \gcd(\dlm-\dle _i),k\mid N(\dlm)/d}\ell
	_{(\dlm-\dle_i)/k}.
\end{align*}
\end{lemm}

\begin{proof}
 Let us consider the $\ndZ$-action on $\mathbb{X}_{\dlm}$ given by
   $$1\cdot i_1i_2\cdots i_m=i_1i_3i_4\cdots i_mi_2.$$
Then $(m-1)\cdot i_1i_2\cdots i_m=i_1i_2\cdots i_m$.
In any $\ndZ$-orbit of $\bX_{\dlm}$
there is a unique element $jv^k$, where $j\in I$, $v$ is a Lyndon word,
and $k\ge 1$. Moreover, then $k\mid m_j-1$ and $k\mid m_t$ for each $1\le t\le n$
with $t\ne j$.

Again, to any $\ndZ$-orbit $\cO $ we attach the submodule $\overline{V}_{\cO}$
of $\overline{V}_{\dlm}$
generated by the monomials $x_{i_1}\cdots x_{i_{m}}$, where $i_1\cdots
i_{m}\in \mathcal{O}$.
Then
\begin{align}
\overline{V}_{\dlm} =\bigoplus_{\cO} \overline{V}_{\cO}.
\end{align}
Let $v=i_1i_2\cdots i_l$ be a Lyndon word, $j\in \{1,\dots,n\}$, and $k\ge 1$.
Assume that $\deg jv^k=\dlm$. Then $l=(m-1)/k$.
Let $\widetilde{v}=x_{i_1}\cdots x_{i_l}\in T(\overline{V})$.
Then the monomials
\begin{align} \label{eq:O2basis}
  x_jx_{i_t}x_{i_{t+1}}\cdots x_{i_l}\widetilde{v}^{k-1}x_{i_1}\cdots x_{i_{t-1}},
\quad 1\le t\le l,
\end{align}
form a basis of $\overline{V}_{\cO}$ for the $\ndZ$-orbit $\cO$ of $jv^k$,
where $x_{i_1}\cdots x_{i_0}=1$.

For any $1\le t\le l$ one obtains that
\begin{align*}
&\overline{\rho}_m(1-\sigma_{m-1}\cdots\sigma_2\sigma_1^2)
(x_jx_{i_t}\cdots x_{i_{l}}\widetilde{v}^{k-1}x_{i_1}\cdots x_{i_{t-1}})\\
&=x_jx_{i_t}\cdots x_{i_{l}}\widetilde{v}^{k-1}x_{i_1}\cdots x_{i_{t-1}}
    -\overline{q}_{ji_t}\overline{q}_{i_tj}\lambda_tx_jx_{i_{t+1}}\cdots
    x_{i_{l}}\widetilde{v}^{k-1}x_{i_1}\cdots x_{i_{t}},
\end{align*}
where $x_{i_{l+1}}\cdots x_{i_l}=1$ and
$\lambda_t=\overline{q}_{i_ti_1}^{k}\overline{q}_{i_ti_2}^{k}\cdots
  \overline{q}_{i_ti_{t-1}}^{k}\overline{q}_{i_ti_{t}}^{k-1}
  \overline{q}_{i_ti_{t+1}}^{k}\cdots \overline{q}_{i_ti_{l}}^{k}$
   for all $1\le t\le l$.
Thus the matrix of $\overline{\rho}_m(1-\sigma_{m-1}\cdots\sigma_2\sigma_1^2)|\overline{V}_{\cO}$
with respect to the basis \eqref{eq:O2basis} is $B=(b_{st})_{1\le s,t\le l}$, where
$$b_{st}=
\begin{cases}
     1 & \text{if $s=t$,}  \\
     -\overline{q}_{ji_t}\overline{q}_{i_tj}\lambda_t & \text{if $s=t+1$, $1\le t\le l-1$,} \\
     -\overline{q}_{ji_l}\overline{q}_{i_lj}\lambda_l & \text{if $s=1$, $t=l$,} \\
     0 & \text{otherwise.}
\end{cases}
$$
Hence,
\begin{align*}
&\det(\overline{\rho}_m(1-\sigma_{m-1}\cdots\sigma_2\sigma_1^2)\big|\overline{V}_{\cO})\\
&=1+(-1)^{l+1}(-1)^l\overline{q}_{ji_1}\overline{q}_{i_1j}\cdots
\overline{q}_{ji_l}\overline{q}_{i_lj}\lambda_1\lambda_2\cdots\lambda_l\\
&=1-\prod_{1\le t \le l}(\overline{q}_{ji_t}\overline{q}_{i_tj})
\prod_{1\le t\le l}\overline{q}_{i_ti_t}^{k-1}
\prod_{1\le t<s\le l}(\overline{q}_{i_ti_s}\overline{q}_{i_si_t})^k\\
&=1-\prod_{1\le t \le n}\overline{q}_{tt}^{m_t(m_t-1)/k}\prod_{1\le t<s\le n}(\overline{q}_{st}\overline{q}_{ts})^{m_tm_s/k}\\
&=1-f(Q_{\dlm})^{N(\dlm)/k}.
\end{align*}
This implies the first claim.

If $R$ is a field, then the matrix $B$ above has corank $0$ or $1$.
Moreover, $B$ has corank $1$ if and only if
$f(Q_{\dlm})^{N(\dlm)/k}=1$, that is, if and only if $d\mid N(\dlm)/k$,
where $d=\mathrm{ord}(f(Q_{\dlm}))$. This implies the last claim.
\end{proof}

\begin{lemm}\label{le:detshuffle}
Let $\dlm\in \ndN_0^n$ with $m=|\dlm|\ge 2$. Then
\begin{align*}
	&\det (\overline{\rho}_m(S_{1,m-1})|\overline{V}_{\dlm})
	\prod_{k\mid \gcd(\dlm)}(1-f(Q_{\dlm})^{N(\dlm)/k})^{\ell_{\dlm/k}}\\
	&=\prod_{i:m_i>0}\Big(\det
	(\overline{\rho}_{m-1}(S_{1,m-2})|\overline{V}_{\dlm-\dle_i})\\
  &\qquad \quad
   \prod_{k\mid \gcd(\dlm-\dle_i)}(1-f(Q_{\dlm})^{N(\dlm)/k})^{\ell_{(\dlm-\dle_i)/k}}
   \Big).
\end{align*}
\end{lemm}

\begin{proof}
From Lemma \ref{basic Lemm} we conclude that
\begin{align*}
&\det (\overline{\rho}_m(S_{1,m-1})|\overline{V}_{\underline{m}})
  \det(\overline{\rho}_m(1-\sigma_{m-1}\cdots \sigma_{2}\sigma_{1})|\overline{V}_{\dlm})\\
  &=\det (\overline{\rho}_m(S_{1,m-2})|\overline{V}_{\underline{m}})
  \det(\overline{\rho}_m(1-\sigma_{m-1}\cdots \sigma_{2}\sigma_{1}^2)|\overline{V}_{\dlm}).
\end{align*}
Because of Lemma \ref{det1} and \ref{det2} the above equality is equivalent to
\begin{align*}
&\det (\overline{\rho}_m(S_{1,m-1})|\overline{V}_{\dlm})
  \prod_{k\mid \gcd(\dlm)}(1-f(Q_{\dlm})^{N(\dlm)/k})^{\ell_{\dlm/k}}\\
  &=\det (\overline{\rho}_m(S_{1,m-2})|\overline{V}_{\dlm})
 \prod_{i:m_i>0}\prod_{k\mid \gcd(\dlm-\dle_i)}(1-f(Q_{\dlm})^{N(\dlm)/k})^{\ell_{(\dlm-\dle_i)/k}}.
\end{align*}
This implies the lemma.
\end{proof}

\begin{prop}\label{pr:detshuffle2}
 Let $\dlm =(m_1,\dots,m_n)\in \ndN_0^n$ with $m=|\dlm|\ge 2$.
\begin{enumerate}
  \item If $\dlm=m_i\dle_i$ with $1\le i\le n$
  then $\overline{\rho}_m(S_{1,m-1})|\overline{V}_{\dlm}=(m_i)_{\overline{q}_{ii}}\id$.
  \item If there exist $1\le i<j\le n$ with $m_i,m_j\ne0$, then
\begin{align}
\label{detshuf}
\det (\overline{\rho}_m(S_{1,m-1})|\overline{V}_{\underline{m}})
	=f(A_{\dlm})\prod_{i:m_i>0}\det (\overline{\rho}_{m-1}(S_{1,m-2})|\overline{V}_{\dlm-\dle_i}).
\end{align}
\end{enumerate}
\end{prop}

\begin{proof}
Claim (1) follows directly from the definition of $S_{1,m-1}$.

In order to prove part (2) of the Proposition it suffices to
consider $R=\ndZ [p_{ij}\mid 1\le i,j\le n]$ and $f=\id $.
In this case the claim follows from
Lemma~\ref{le:detshuffle} and Lemma~\ref{le:Am}.
\end{proof}

\section{Nichols algebras which are free algebras}
\label{freeness}

In the remaining part of this paper let $\fie$ be a field,
let $\fie^{\times}=\fie\backslash \{0\}$,
and let $(V,c)$ be an $n$-dimensional braided vector space of diagonal type
with basis $x_1,x_2, \ldots, x_n$
and braiding matrix $\bq \in (\fie^{\times})^{n\times n}$.
Let $T(V)$ and $\NA(V)$ denote the tensor algebra and the Nichols algebra of $V$, respectively.

For the basic theory of Nichols algebras we refer to \cite{AS2002}.

In this section we determine when $\NA (V)$ is a free algebra, that is, $\NA (V)=T(V)$.

For all $k\ge 2$ there is a unique group homomorphism $\tau:\BG k\to \BG{k+1}$
with $\tau(\sigma_i)=\sigma_{i+1}$ for all $1\le i<k$. We also write $\tau $
for the induced algebra maps $\fie \BG k\to \fie \BG{k+1}$.

For all $m\ge 2$ let
$\rho_m:\fie\BG m\longrightarrow \End(V^{\otimes m})$
be the representation of $\fie\BG m$ introduced in Section \ref{shuffles} as $\overline{\rho}_m$,
and let
$$ S_m=S_{1,m-1}\tau(S_{1,m-2})\tau^2(S_{1,m-3})\cdots \tau^{m-2}(S_{1,1})\in \fie\BG m.$$
Then, by \cite{schauenburg1996},
\begin{align} \label{schauenburg}
\NA(V)=\fie \oplus V \oplus \bigoplus_{m=2}^{\infty} V^{\ot m}/\ker(\rho _m(S_m)).
\end{align}

\begin{lemm}\label{det=0}
Let $\dlm\in \ndN_0^n $ with $m=|\dlm|\ge 2$.

(1) If $P_{\dlm}(\bq)=0$, then $\det (\rho_m(S_{1,m-1})|V_{\dlm})=0$.

(2) If $\det (\rho_m(S_{1,m-2})|V_{\dlm})\ne0$ and
$\det (\rho_m(S_{1,m-1})|V_{\dlm})=0$
then $P_{\dlm}(\bq)=0$.
\end{lemm}

\begin{proof}
 If $\dlm=m_i\dle_i$ for some $1\le i\le n$, $m_i\in\ndN$,
 then the claim holds because of Proposition~\ref{pr:detshuffle2}(1).

 Assume now that there exist $1\le i<j\le n$ with $m_i,m_j\ne0$.
 Then Proposition~\ref{pr:Am=0} implies that
 $A_m(\bq)=0$ if and only if $P_{\dlm}(\bq)=0$.
 Hence the lemma follows from Proposition~\ref{pr:detshuffle2}(2).
\end{proof}

\begin{prop}\label{prop:main}
Let $\dlm\in \ndN_0^n $ with $m=|\dlm|\ge 2$.
\begin{enumerate}
   \item If $P_{\dlm}(\bq)=0$, then there is a non-trivial relation in
   $\NA(V)$ of degree $\dlm$.
   \item If $P_{\dll}(\bq)\ne0$ for all $\dll\le \dlm$ with $|\dll|\ge 2$,
   then there is no non-trivial relation
   in $\NA(V)$ of degree $\dlm$.
\end{enumerate}
\end{prop}

\begin{proof}
(1) Assume that $P_{\dlm}(\bq)=0$.
Then $\det (\rho_m(S_{1,m-1})|V_{\dlm})=0$ by Lemma \ref{det=0}(1). Thus
$\det (\rho_m(S_{m})|V_{\dlm})=0$ by the definition of $S_{m}$,
and the claim follows from Equation~\eqref{schauenburg}.

(2) Assume that the Nichols algebra
$\NA (V)$ has a non-trivial relation in degree $\dlm $.
Let $\dll \le \dlm $ be such that $\NA (V)$ has a
non-trivial relation in degree $\dll $ and no non-trivial relation in
any degree $< \dll $. Let $l=|\dll |$. Then $l\ge 2$,
$\ker (\rho_l(S_{1,l-2})|V_{\dll})=0$,
and $\ker (\rho_l(S_{1,l-1})|V_{\dll})\ne 0$
by \eqref{schauenburg} and by the definition of $S_l$.
Hence $P_{\dll}(\bq)=0$ by Lemma~\ref{det=0}(2). This proves (2).
\end{proof}

Based on the above proposition we obtain our first main Theorem as follows.
\begin{theo}\label{theo}
We have $\NA(V)=T(V)$ if and only if $P_{\dlm}(\bq)\ne0$
for all $\dlm\in \ndN^n$ with $|\dlm|\ge 2$.
\end{theo}

\begin{proof}
 The claim follows immediately from Proposition~\ref{prop:main}.
\end{proof}

\begin{exam}(Diophantine equation)\label{Diophequation}
Assume that the characteristic of $\fie $ is neither $2$ nor $3$
and that $\bq=(q^{a_{ij}})_{i,j\in I}$ with $q\in \fie ^\times $ not a root of $1$
and $(a_{ij})_{i,j\in I}\in \ndZ^{n\times n}$.
For any $\dlm =(m_1,\dots,m_n)\in \ndN_0^n$ let
$$K(\dlm)=\sum_{i,j=1}^{n}a_{ij}m_im_j,\quad  \lambda(\dlm)=\sum_{i=1}^{n}a_{ii}m_i.$$
Then $P_{\dlm}(\bq)\ne 0$ in the following cases:
\begin{enumerate}
\item $\dlm=m\dle_i$, $m\ge 2$, $1\le i\le n$,
\item $\dlm=2\dle_i+2m\dle_j$, $m\ge 1$, $i,j\in I$, $i\ne j$,
\item $\dlm=3\dle_i+3\dle_j$ or $\dlm=4\dle_i+4\dle_j$, $i,j\in I$, $i\ne j$.
\end{enumerate}
Moreover, for any other $\dlm\in \ndN_0^n$ with $|\dlm|\ge 2$,
$$ P_{\dlm}(\bq)=0 \text{ if and only if }K(\dlm)=\lambda (\dlm). $$
Hence, by Theorem~\ref{theo},
$\NA(V)=T(V)$ if and only if there is no solution of the diophantine equation
$K(\dlm)=\lambda(\dlm)$
with $\dlm\in \ndN_0^n$, $|\dlm|\ge 2$,
$\dlm \notin \{m\dle_i\mid m\ge 2\}\cup \{2\dle_i+2m\dle_j,3\dle_i+3\dle_j,4\dle_i+4\dle_j
\mid m\ge 1,i,j\in I, i\ne j\}$.

We now provide concrete examples of Nichols algebras of diagonal type which are identified
by this example as a free algebra.
Let $n=2$ and let $a,b$ be positive integers with $a>b$.
Let $q\in \fie^\times $ be not a root of $1$ and let
$\bq=(q_{ij})_{1\le i,j\le 2}$ with $q_{11}=q_{22}=q^a$ and $q_{12},q_{21}\in q^{\ndZ}$
with $q_{12}q_{21}=q^{-b}$. For any $\dlm=(m_1,m_2)\in \ndN_0^2$ with $|\dlm|\ge2$
we have
$$K(\dlm)=am_1^2-bm_1m_2+am_2^2,\quad \lambda(\dlm)=am_1+am_2$$
and hence
\begin{align*}
 &K(\dlm)-\lambda(\dlm)=am_1^2-bm_1m_2+am_2^2-(am_1+am_2)\\
 &\quad =a(m_1-m_2)(m_1-m_2-1)+(2a-b)m_2\Big(m_1-2+\frac{2(a-b)}{2a-b}\Big).
\end{align*}
Assume that $K(\dlm)=\lambda (\dlm)$ and $m_1,m_2>0$. By symmetry of $m_1,m_2$,
without loss of generality we may assume that $m_1\ge m_2$.
Then
$$2(m_1-m_2)(m_1-m_2-1)\ge 0, \quad m_2>0, $$
and
$m_1-2+\frac{2(a-b)}{2a-b}>m_1-2\ge 0$ for $m_1\ge 2$ since $a>b$.
Hence $K(\dlm)=\lambda(\dlm)$ implies that $m_1=1$ and hence $m_2=1$. However,
$K(1,1)-\lambda(1,1)=-b\ne 0$. Therefore $K(\dlm)\ne \lambda (\dlm)$
for all pairs $\dlm=(m_1,m_2)$ with $m_1,m_2>0$. Thus $\NA (V)=T(V)$.
\end{exam}

\section{An upper bound on the dimension of the kernel of the shuffle map}
\label{upperbound}

Let $n\in \ndN$, $I=\{1,2,\dots,n\}$, and let $R=\ndZ [\bar q_{ij}^{\pm1}\mid i,j\in I]$.
Let $(\overline{V},\overline{c})$ be the free prebraided module of diagonal type over $R$
with basis $\bar x_1,\bar x_2,\ldots,\bar x_n$  and braiding matrix
$\bar \bq=(\bar q_{ij})_{i,j\in I}$
such that $$\overline{c}(\bar x_i\otimes \bar x_j)=\bar q_{ij}\bar x_j\otimes \bar x_i.$$

Assume that $\mathrm{char}(\fie)=0$.
Let $(V,c)$ be an $n$-dimensional braided vector space over $\fie $
with braiding matrix $\bq=(q_{ij})_{i,j\in I}$ and with basis $x_1,x_2,\ldots,x_n$
such that
$$c(x_i\otimes x_j)=q_{ij}x_j\otimes x_i \quad \text{for all $i,j\in I$.}$$
There are unique ring homomorphisms
\begin{align*}
 \eta&:R\to \fie ,\\
 \eta'&:\ndZ[p_{ij}\mid 1\le i,j\le n]\to R,\\
 \eta''=\eta \eta'&:\ndZ[p_{ij}\mid 1\le i,j\le n]\to \fie
\end{align*}
with $\eta(\bar q_{ij})=q_{ij}$, $\eta'(p_{ij})=\bar q_{ij}$
for all $i,j\in I$. We view them as evaluation at $\bq$, $\bar \bq$, and $\bq$, respectively.
Correspondingly, we write
$$ \eta(\bar p)=\bar p(\bq ), \quad \eta'(p)=p(\bar \bq), \quad \eta''(p)=p(\bq) $$
for any $p\in \ndZ[p_{ij}\mid 1\le i,j\le n]$ and any $\bar p\in R$.

Let $(\alpha_{ij})_{i,j\in I}$ be a basis of $\ndZ^{n\times n}$.
For any $\alpha =\sum_{i,j\in I}a_{ij}\alpha_{ij}$ let
$$ \bar q_{\alpha}=\prod_{i,j\in I}\bar q_{ij}^{a_{ij}}\in R,\quad
   q_{\alpha}=\bar q_\alpha(\bq)=\prod_{i,j\in I}q_{ij}^{a_{ij}}\in \fie^\times.
$$

\begin{lemm}\label{ring morph}
 Let $\dlm\in \ndN_0^n$ with $|\dlm|\ge2$.
 Let $\eta_1:\fie[t,t^{-1}]\rightarrow\fie$ be the ring homomorphism given by
 $$ \eta_1(t)=Q_{\dlm}(\bq).$$
 Then there exists a ring homomorphism $\eta_2:R\rightarrow \fie[t,t^{-1}]$
 such that
 $$ \eta_2(Q_{\dlm}(\bar \bq ))=t, \quad \eta_1\eta_2=\eta. $$
\end{lemm}

\begin{proof}
By Remark~\ref{Qfactor} and Lemma~\ref{le:Laurentauto},
there exists a ring automorphism $\varphi $ of $R$ with
$$\varphi (\bar q_{11})
=\prod_{i=1}^n\bar q_{ii}^{m_i(m_i-1)}\prod_{1\le i<j\le n}(\bar q_{ij}\bar q_{ji})^{m_im_j}.$$
Let $\eta'_2:R\to \fie [t,t^{-1}]$ be the ring homomorphism with
$\eta'_2(\bar q_{11})=t$, $\eta'_2(\bar q_{ij})=\varphi(\bar q_{ij})(\bq)$
for all $i,j\in I$ with $(i,j)\ne (1,1)$.
Then $$\eta_2=\eta'_2\varphi ^{-1}:R\to \fie[t,t^{-1}]$$
is a ring homomorphism and
\begin{align*}
\eta_2(Q_{\dlm}(\bar \bq))&=\eta'_2\varphi^{-1}\varphi(\bar q_{11})=t,\\
\eta_1\eta_2(\varphi (\bar q_{11}))&=\eta_1\eta_2(Q_{\dlm}(\bar q))=\eta_1(t)=Q_{\dlm}(\bq),\\
\eta_1\eta_2(\varphi(\bar q_{ij}))&=\eta_1\eta'_2(\bar q_{ij})
=\eta_1(\varphi(\bar q_{ij})(\bq ))=\varphi(\bar q_{ij})(\bq )
\end{align*}
for all $i,j\in I$ with $(i,j)\ne (1,1)$.
Thus $\eta_1\eta_2=\eta$.
\end{proof}

\begin{lemm}\label{le:irrf}
Let $\dlm\in \ndN_0^n$ with $|\dlm |\ge 2$. Assume that $P_{\dlm}(\bq)=0$.
Then $Q_{\dlm}(\bq)^{N(\dlm)}=1$. Let $d=\ord(Q_{\dlm}(\bq))$.
Then $\Phi_d(Q_{\dlm})$
is the unique irreducible factor $f$ of $P_{\dlm}\in \ndZ[p_{ij}|i,j\in I]$
such that $f(\bq)=0$.
\end{lemm}

\begin{proof}
	The claim follows directly from Remark~\ref{Qfactor} and
	Lemma~\ref{le:PhikQirred}.
\end{proof}

\begin{lemm}\label{le:n1n2}
Let $\dlm\in \ndN^n_0$ with $m=|\dlm|\ge 2$.
Assume that $P_{\dlm}(\bq)=0$. Let $d=\ord (Q_{\dlm}(\bq))$ and $d'=N(\dlm)/d$.
Then
\begin{enumerate}
\item $\Phi_d(Q_{\dlm}(\bar \bq))$ does not appear in the prime decomposition of
the polynomial $\det(\overline{\rho}_m(S_{1,m-2})|\olV_{\dlm})\in R$, and
\item
$\Phi_d(Q_{\dlm}(\bar \bq))$ appears
$n_1(\bq)-n_2(\bq)$ times in the prime decomposition of
$\det(\overline{\rho}_m(S_{1,m-1})|\olV_{\dlm})$,
where $$n_1(\bq)=\sum_{i:m_i>0}\sum_{k\mid \gcd(\dlm-\dle_i),k\mid d'}\ell_{(\dlm-\dle_i)/k},$$
and
$$n_2(\bq)=\sum_{k\mid \gcd(\dlm),k\mid d'}\ell_{\dlm/k}.$$
\end{enumerate}
\end{lemm}

\begin{proof}
Assume first that $\dlm=m\dle_i$ for some $i\in I$. Then $Q_{\dlm}=p_{ii}$,
$P_{\dlm}=(m)_{p_{ii}}$, $N(\dlm)=m(m-1)$, $q_{ii}\ne 1$, and $d\mid m$, $d>1$.
Hence $\gcd(m-1,d')=m-1$, $\gcd(m,d')=m/d$, and therefore
$$ \det(\overline{\rho}_m(S_{1,m-1})|\olV_{\dlm})=(m)_{\bar q_{ii}},
\quad n_1(\bq)=1,\quad n_2(\bq)=0 $$
by Proposition~\ref{pr:detshuffle2}(1)
and Remark~\ref{re:smallell}. Thus (2) holds in this case. Moreover,
(1) is valid since $\det(\overline{\rho}_m(S_{1,m-2})|\olV_{\dlm})=(m-1)_{\bar q_{ii}}$
and $d$ does not divide $m-1$.

Assume that there exists $1\le i<j\le n$ with $m_i,m_j\ne 0$.
Since $P_{k\dle_i}=\det(\overline{\rho}_k(S_{1,k-1})|\olV_{k\dle_i})$
for all $k\ge 2$ and $i\in I$, Propositions~\ref{pr:detshuffle2}(2) and
\ref{pr:Am=0} imply that any irreducible factor of
$\det(\overline{\rho}_{m-1}(S_{1,m-2})|\olV_{\dlm})$
is an irreducible factor of some $P_{\dll}$ with $\dll <\dlm$.
Thus, by Proposition~\ref{pr:Am=0} and Lemma~\ref{le:prime},
$\Phi_d(Q_{\dlm}(\bar \bq))$ and
$\det(\overline{\rho}_{m-1}(S_{1,m-2})|\olV_{\dlm})$
are relatively prime, which proves (1). By Proposition~\ref{pr:detshuffle2}(2),
for the proof of (2) it remains to determine the
multiplicity of the irreducible factor $\Phi_d(Q_{\dlm})$ in $A_{\dlm }$.
Let $k\in\ndN$ with $k|N(\dlm)$. By Lemma~\ref{le:PhikQirred},
$\Phi_d(Q_{\dlm})$ has multiplicity $0$ in $Q^{N(\dlm)/k}-1$, except when $k\mid d'$,
in which case it has multiplicity $1$.
By Definition~\ref{defAm}, there are
\begin{align*}
n_1(\bq)
=\sum_{i:m_i>0}\sum_{k\mid \gcd(\dlm-\dle_i),k|d'}\ell_{(\dlm-\dle_i)/k}
\end{align*}
factors $\Phi_d(Q_{\dlm})$ in the numerator of $A_{\dlm}$
and
\begin{align*}
n_2(\bq)=\sum_{k\mid \gcd(\dlm),k|d'}\ell_{\dlm/k}
\end{align*}
factors $\Phi_d(Q_{\dlm})$ in the denominator of $A_{\dlm}$.
This proves the lemma.
\end{proof}

\begin{lemm}\label{le:rmul1}
Let $\dlm\in \ndN^n_0$ with $m=|\dlm|\ge 2$.
Suppose that $P_{\dlm}(\bq)=0$ and $P_{\dll}(\bq)\ne0$ for all $\dll<\dlm$.
Let
$\eta_1:\fie [t,t^{-1}]\to \fie $ and $\eta_2:R\to \fie [t,t^{-1}]$ be ring
homomorphisms as in
Lemma~\ref{ring morph}. Then the following hold.
\begin{enumerate}
	\item $\eta_1\eta_2(\det(\overline{\rho}_m(S_{1,m-2})|\olV_{\dlm}))\ne0$ in
		$\fie $,
	\item
		the polynomials $t-Q_{\dlm}(\bq)$ and
		$\eta_2(\det(\overline{\rho}_m(S_{1,m-2})|\olV_{\dlm}))$
		are relatively prime in $\fie [t,t^{-1}]$, and
	\item the factor $t-Q_{\dlm}(\bq)\in \fie [t,t^{-1}]$ appears
		$n_1(\bq)-n_2(\bq)$ times in the prime decomposition of
		$\eta_2(\det(\overline{\rho}_m(S_{1,m-1})|\olV_{\dlm}))$.
\end{enumerate}
\end{lemm}

\begin{proof}
(1) By Lemma \ref{ring morph}, we have
\begin{align*}
	\eta_1\eta_2(\det(\overline{\rho}_m(S_{1,m-2})|\olV_{\dlm}))
	=&\eta(\det(\overline{\rho}_m(S_{1,m-2})|\olV_{\dlm}))\\
	=&\det(\rho_m(S_{1,m-2})|V_{\dlm}).
\end{align*}
As $P_{\dll}(\bq)\ne0$ for all $\dll<\dlm$, we get
$\det(\rho_m(S_{1,m-2})|V_{\dlm})\in \fie^{\times}$ because of
Proposition~\ref{prop:main}.

(2) By definition, $\eta_1(t-Q_{\dlm}(\bq))=0$. Thus (2) follows from (1).

(3) Assume first that $\dlm=m\dle_i$ for some $i\in I$. Then
$Q_{\dlm}=p_{ii}$, $\overline{\rho}_m(S_{1,m-1})|\olV_{\dlm}=(m)_{\bar q_{ii}}\id
$ by Proposition~\ref{pr:detshuffle2}(1), and hence
$$\eta_2(\det (\overline{\rho}_m(S_{1,m-1})|\olV_{\dlm}))=(m)_t. $$
Since $(m)_{q_{ii}}=0$ in $\fie $ and $\mathrm{char}(\fie )=0$, the irreducible
factor $t-q_{ii}$ appears once in $(m)_t$. Moreover, $n_1(\bq)=1$ and
$n_2(\bq)=0$, see the first part of the proof of Lemma~\ref{le:n1n2}.

Assume now that there exist $1\le i<j\le n$ with $m_i,m_j\ne 0$.
By (1),
$t-Q_{\dlm}(\bq)$ does not appear in the prime decomposition of
$\eta_2(\det(\overline{\rho}_m(S_{1,m-2})|\olV_{\dlm}))$.
Hence, by Proposition~\ref{pr:detshuffle2}(2), we have to determine the
multiplicity $M$ of $t-Q_{\dlm}(\bq)$ in the prime decomposition of
$\eta_2(A_{\dlm}(\bar \bq))$.
By the definition of $A_{\dlm}$ and by Lemma~\ref{le:Am},
$\eta_2(A_{\dlm}(\bar \bq))$ is a non-zero polynomial in $\fie[t,t^{-1}]$.
By Remark~\ref{Qfactor} and by Proposition~\ref{pr:Am=0},
$A_{\dlm}$ is a product of polynomials of the
form $\Phi _k(Q_{\dlm})$ with $k\mid N(\dlm)$. Hence
$\eta_2(A_{\dlm}(\bar \bq))$ is a
product of polynomials of the form $\Phi _k(t)$ with $k\mid N(\dlm)$,
and the multiplicity of $\Phi_k(t)$ with $k\mid N(\dlm)$ in
$\eta_2(A_{\dlm}(\bar \bq ))$ is the same
as the multiplicity of $\Phi_k(Q_{\dlm})$ in $A_{\dlm}$.
Let $d=\ord (Q_{\dlm}(\bq))$.
Then $t-Q_{\dlm}(\bq)$
divides $\Phi _k(t)$ if and only if $k=d$.
Hence $M$ is the multiplicity of $\Phi _d(Q_{\dlm})$ in $A_{\dlm}$.
Therefore, by Proposition~\ref{pr:detshuffle2}(2) and by Lemma~\ref{le:n1n2},
$M$ is the multiplicity of $\Phi_d(Q_{\dlm}(\bar \bq))$
in $\det(\overline{\rho}_m(S_{1,m-1})|\olV_{\dlm})$, that is, $M=n_1(\bq)-n_2(\bq)$.
\end{proof}


\begin{prop}\label{prop upperbound}
Let $\dlm\in \ndN^n_0$ with $m=|\dlm|\ge2$.
Suppose that $P_{\dlm}(\bq)=0$ and that $P_{\dll}(\bq)\ne0$ for all $\dll<\dlm$ with $|\dll|\ge 2$.
Then
$$\dim (\ker(\rho_m(S_{1,m-1})|V_{\dlm}))\le n_1(\bq)-n_2(\bq).$$
\end{prop}

\begin{proof}
%
Let $\eta_1:\fie [t,t^{-1}]\to \fie $ and $\eta_2:R\to \fie [t,t^{-1}]$ be ring homomorphisms
as in Lemma~\ref{ring morph}.
Let $M$ be the Smith normal form of $\eta_2(\overline{\rho}_m(S_{1,m-1})|\olV_{\dlm})$,
which is a diagonal matrix.
Then $\det(M)=\det (\eta_2(\overline{\rho}_m(S_{1,m-1})|\olV_{\dlm}))
\ne 0$ by Lemma~\ref{le:rmul1}(3), and hence
there is no zero on the diagonal of $M$. Again by Lemma~\ref{le:rmul1}(3),
$t-Q_{\dlm}(\bq)$ appears $n_1(\bq)-n_2(\bq)$ times in the prime decomposition of
$\det(M)$. Hence $t-Q_{\dlm}(\bq)$ appears in at most $n_1(\bq)-n_2(\bq)$
diagonal entries of $M$ as an irreducible factor. Then
\begin{align*}
  \dim (\ker (\rho_m(S_{1,m-1})|V_{\dlm}))
  &=\dim (\ker ( \eta_1\eta_2(\overline{\rho}_m(S_{1,m-1})|\olV_{\dlm})))\\
  &=\dim (\ker (\eta_1(M))\le n_1(\bq)-n_2(\bq)).
\end{align*}
Hence the proposition holds.
\end{proof}

\section{The dimension of the kernel of shuffle map}
\label{dim}

We use some notation and conventions from the previous section.
So let us assume that $\mathrm{char}(\fie)=0$. Let $n\in \ndN $, let $\bq=(q_{ij})_{1\le i,j\le n}
\in (\fie ^\times)^{n\times n}$, and let $(V,c)$ be a braided vector space of diagonal type
with basis $x_1,\dots,x_n$ such that $c(x_i\ot x_j)=q_{ij}x_j\ot x_i$ for all $1\le i,j\le n$.
For each $\dlm \in \ndN_0^n$ with $|\dlm|\ge 2$ let $n_1(\bq),n_2(\bq)\ge 0$
be the integers defined in Lemma~\ref{le:n1n2}(2).

 In this section we determine the dimension of the kernel of the shuffle map
 $\rho_{|\dlm|}(S_{1,|\dlm|-1})|V_{\dlm}$ for those
 $\dlm \in \ndN_0^n$ with $|\dlm|\ge 2$, $P_{\dlm}(\bq)=0$,
 and $P_{\dll}(\bq)\ne 0$ for all $\dll<\dlm$ with $|\dll|\ge 2$.

\begin{prop}\label{lower bound}
	Let $\dlm\in \ndN_0^n$ with $m=|\dlm|\ge 2$.
	Suppose that $P_{\dlm}(\bq)=0$ and $P_{\dll}(\bq)\ne0$ for all $\dll<\dlm$
	with $|\dll|\ge 2$. Then
\begin{align*}
	\dim(\ker(\rho_m(S_{1,m-1})|V_{\dlm}))\ge n_1(\bq)-n_2(\bq),
\end{align*}
\end{prop}

\begin{proof}
Since $P_{\dll}(\bq)\ne0$ for all $\dll<\dlm$ with $|\dll|\ge 2$,
it follows from Proposition~\ref{prop:main} that
$\rho_m(S_{1,m-2})|V_{\dlm}$ is injective. Hence
\begin{align*}
&\ker(\rho_m(1-\sigma_{m-1}\cdots\sigma_2\sigma_1)\rho_m(S_{1,m-1})|V_{\dlm})\\
&\qquad =\ker(\rho_m(1-\sigma_{m-1}\cdots\sigma_2\sigma_1^2)|V_{\dlm})
\end{align*}
because of Lemma \ref{basic Lemm}.
Therefore
\begin{align*}
\dim(\ker(\rho_m(S_{1,m-1})|V_{\dlm}))
+\dim(\ker(\rho_m(1-\sigma_{m-1}\cdots\sigma_{2}\sigma_{1})| V_{\dlm}))\\
\ge \dim(\ker(\rho_m(1-\sigma_{m-1}\cdots\sigma_{2}\sigma_{1}^2)|V_{\dlm})),
\end{align*}
that is,
\begin{align*}
\dim(\ker(\rho_m(S_{1,m-1})|V_{\dlm}))
+n_2(\bq)\ge n_1(\bq)
\end{align*}
because of Lemmas~\ref{det1} and \ref{det2}.
This proves the Proposition.
\end{proof}

\begin{theo}\label{theo dim}
Let $\dlm\in \ndN_0^n$ with $m=|\dlm|\ge 2$. Assume that $P_{\dlm}(\bq)=0$
and $P_{\dll}(\bq)\ne0$ for all $\dll<\dlm$ with $|\dll|\ge 2$. Then
\begin{align*}
\dim(\ker(\rho_m(S_{1,m-1})|V_{\dlm}))=n_1(\bq)-n_2(\bq).
\end{align*}
\end{theo}

\begin{proof}
The theorem follows immediately from Propositions \ref{prop upperbound} and Proposition \ref{lower bound}.
\end{proof}

\begin{rema}
  Let $\dlm\in \ndN_0^n$ with $m=|\dlm|\ge 2$. Assume 
  $\ker (\rho _k(S_k)|V_{\dll})=0$ for all $\dll <\dlm $ with $k=|\dll |\ge 2$.
  Then
  $\ker (\rho _k(S_{1,k-1})|V_{\dll})=0$ for all $\dll <\dlm $ with $k=|\dll |\ge 2$
  by the definition of $S_k$. Moreover,
  $P_{\dll}(\bq)\ne 0$ for all $\dll <\dlm $ with $|\dll|\ge 2$ by
  Proposition~\ref{prop:main}(1). Assume now that $\dim (\ker (\rho _m(S_m)|V_{\dlm}))>0$.
  Then $P_{\dlm}(\bq)=0$ by Proposition~\ref{prop:main}(2).
  Thus
  $$\dim (\ker (\rho _m(S_m)|V_{\dlm}))=n_1(\bq)-n_2(\bq)$$
  by Theorem~\ref{theo dim} and by the bijectivity of the maps
  $\rho _k(S_{1,k-1})|V_{\dll}$ for $\dll <\dlm $ with $k=|\dll |\ge 2$.
\end{rema}

\begin{exam}
  Here we give an example of a Nichols algebra of diagonal type
  where in some degree one has two defining relations.

  Let $(V,c)$ be the two-dimensional braided vector space of diagonal type
  with basis $x_1, x_2$ and braiding matrix
  $\bq=(q_{ij})_{1\le i,j\le 2}\in (\fie^{\times})^2$, such that
   $$c(x_i\ot x_j)=q_{ij}x_j\ot x_i$$
  and $q_{11}=q_{22}=q_{12}q_{21}=q$,
  where $q\in \fie^{\times}$ is a primitive fifth root of unity.

   Let $\dlm=3\dle_1+4\dle_2$. Then $N(\dlm)=\gcd(6,12,12)=6$ and
   $$ Q_{\dlm}=p_{11}(p_{12}p_{21})^2p_{22}^2,\qquad Q_{\dlm}(\bq)=q^5=1. $$
   Thus $d=\ord (Q_{\dlm})=1$ and $d'=N(\dlm)/d=6$.

  Let us check that $P_{\dll}(\bq)\ne 0$ for all $\dll<\dlm$ with $|\dll|\ge 2$.
  \begin{align*}
   &P_{(3,3)}(\bq)=(3)_{q_{11}^2(q_{12}q_{21})^3q_{22}^2}=(3)_{q^7}=(3)_q(3)_{-q};\\
   &P_{(3,2)}(\bq)=P_{(2,3)}(\bq)=1-q_{11}(q_{12}q_{21})^3q_{22}^3=1-q^7=1-q^2;\\
   &P_{(3,1)}(\bq)=P_{(1,3)}(\bq)=1-q_{12}q_{21}q_{22}^2=1-q^3;\\
   &P_{(1,2)}(\bq)=P_{(2,1)}(\bq)=1-q_{11}q_{12}q_{21}=1-q^2;\\
   &P_{(1,1)}(\bq)=1-q_{11}=1-q;\\
   &P_{(2,4)}(\bq)=1-q_{22}^6(q_{12}q_{21})^4q_{11}=1-q^{11}=1-q;\\
   &P_{(1,4)}(\bq)=1-q_{22}^3q_{12}q_{21}=1-q^4;\\
   &P_{(2,2)}(\bq)=1+q_{22}(q_{12}q_{21})^2q_{11}=1+q^4;
  \end{align*}
  Moreover, $P_{(m,0)}(\bq),P_{(0,m)}(\bq)\ne 0$ for $m\in \{2,3,4\}$
  since $(4)_q^!\ne 0$. Hence $P_{\dll}(\bq)\ne 0$ for all $\dll <\dlm$ with
  $|\dll|\ge 2$.

  We now calculate $n_1(\bq)$ and $n_2(\bq)$. By definition,
 \begin{align*}
 n_1(\bq)&=\sum_{k|\gcd(2,4),k|6}\ell_{(2,4)/k}+
           \sum_{k|\gcd(3,3),k|6}\ell_{(3,3)/k}\\
        &=\ell_{(2,4)}+\ell_{(1,2)}+\ell_{(3,3)}+\ell_{(1,1)}=7
 \end{align*}
 since $\ell_{(2,4)}=2$, $\ell_{(3,3)}=3$, $\ell_{(1,2)}=\ell_{(1,1)}=1$.
 \begin{align*}
 n_2(\bq)&=\sum_{k|\gcd(3,4),k|6}\ell_{(3,4)/k}=\ell_{(3,4)}=5.
 \end{align*}
 Hence $\dim(\ker(\rho_7(S_{1,6})|V_{(3,4)}))=n_1(\bq)-n_2(\bq)=2$.
\end{exam}


\end{document}